\documentclass[11pt]{article}
\usepackage{amsmath,amsfonts,amssymb,amsthm}
\usepackage{mathrsfs}
\usepackage{latexsym}
\usepackage[english]{babel}

\usepackage{graphicx}
\usepackage[usenames,dvipsnames]{xcolor}

\renewenvironment{proof}[1][\proofname]{{\bfseries #1.} }{\qed}

\def\Cov{{\rm Cov\,}}

\newcommand{\field}[1]{\mathbb{#1}}

\newcommand{\R}{\field{R}}

\newcommand{\Rd}{\field{R}^d}
\newcommand{\N}{\field{N}}

\newcommand{\Z}{\field{Z}}

\newcommand{\Var}{{\rm Var}}
\newcommand{\Corr}{{\rm Corr}}

\newcommand{\eps}{\varepsilon}

\newcommand{\Tb}{{\mathbb{T}}}

\newcommand{\Lc}{{\mathcal{L}}}
\newcommand{\var}{\operatorname{Var}}
\newcommand{\Sc}{\mathcal{S}}
\newcommand{\Nc}{\mathcal{N}}

\def\authors#1{{ \begin{center} #1 \vspace{0pt} \end{center} } \smallskip}
\def\institution#1{{\sl \begin{center} #1 \vspace{0pt} \end{center} } }
\def\inst#1{\unskip $^{#1}$}
\def\title#1{{\huge\bf  \begin{center} #1 \vspace{0pt} \end{center}  } \smallskip}

\def\E{{\mathbb{ E}}}

\def\P{{\mathbb{P}}}

\newtheorem{theorem}{Theorem}[section]

\newtheorem{proposition}[theorem]{Proposition}

\newtheorem{lemma}[theorem]{Lemma}

\newtheorem{definition}[theorem]{Definition}

\newtheorem{remark}[theorem]{Remark}

\renewcommand{\(}{\left(}
\renewcommand{\)}{\right)}
\renewcommand{\[}{\left[}
\renewcommand{\]}{\right]}
\renewcommand{\{}{\left\lbrace}
\renewcommand{\}}{\right\rbrace}

\newcommand{\abs}[1]{\left\lvert#1\right\rvert}

\def\to{\longrightarrow}
\def\sto{\rightarrow}
\def\wh{\widehat}
\def\wt{\widetilde}

\newcommand{\ds}{\displaystyle}

\DeclareMathOperator{\Ima}{Im}

\allowdisplaybreaks

\begin{document}

\date{Jun 2020}

\title{\sc Non-Universal Moderate Deviation Principle for the Nodal Length of Arithmetic Random Waves}
\authors{\large Claudio Macci\inst{1}, Maurizia Rossi\inst{2}, Anna Vidotto\inst{3}}
\institution{\inst{1}Dipartimento di Matematica, Universit\`a di Roma ``Tor Vergata"\\
\inst{2}Dipartimento di Matematica e Applicazioni, Universit\`a di Milano-Bicocca\\
\inst{3}Dipartimento di Matematica e Applicazioni, Università degli Studi di Napoli Federico II}

\begin{abstract}

Inspired by the recent work \cite{MRT:21}, we prove a non-universal non-central Moderate Deviation principle for the nodal length of arithmetic random waves (Gaussian Laplace eigenfunctions on the standard flat torus) both on the whole manifold and on shrinking toral domains. Second order fluctuations for the latter were established in \cite{MPRW:16} and \cite{BMW:20} respectively, by means of chaotic expansions, number theoretical estimates and full correlation phenomena. Our proof is simple and relies on the interplay between the long memory behavior of arithmetic random waves and the chaotic expansion of the nodal length, as well as on well-known techniques in Large Deviation theory (the contraction principle and the concept of exponential equivalence).

\smallskip


\smallskip


\end{abstract}


\section{Introduction} 

In recent years there has been a growing interest for the geometry of random waves motivated by theoretical issues in Geometric Analysis, Mathematical Physics, Probability, Number Theory (see among others \cite{Ya:82, Be:77,Be:02, Wi:10, KKW:13}) as well as applications, for instance in Cosmology, Climate Science and  Brain Imaging (see for instance \cite{MP:11,Ch:05, CMV:23, TW:07}). 

In 1982 Yau \cite{Ya:82} conjectured that the nodal volume, i.e., the measure of the zero locus, which is a smooth (hyper)surface outside of a codimension-two singular set, of \emph{any} Laplacian eigenfunction $f$ on a closed $C^\infty$-smooth Riemannian manifold $(M,g)$ is comparable to the square root of the corresponding eigenvalue $E$. More precisely,  
\begin{equation}\label{yau}
c\sqrt{E}\le \textnormal{Vol}_g(f^{-1}(0))\le C\sqrt{E},
\end{equation}
for some constants $C,c>0$ depending only on the manifold.
In 1978 Br\"uning \cite{Br:78} proved the lower bound in \eqref{yau} in dimension two (for the length of nodal lines),  while for real analytic metrics in any dimension this conjecture was settled by Donnelly and Fefferman \cite{DF:88} ten years later.  Recently, Logunov \cite{Lo:18} established the lower bound in the smooth case in any dimension, while the optimal upper estimate in \eqref{yau} is still an open problem in full generality. 

\

Inspired by the work by Kac on zeros of random polynomials, in 1985 B\'erard \cite{Be:85} proposed to investigate  \emph{random} eigenfunctions: for compact symmetric spaces of rank one (the round sphere e.g.) he suggested to make use of the multiplicities in the spectrum of the Laplacian in order to endow the eigenspace (which is a \emph{finite} dimensional vector space),
say of eigenvalue $E$, with a Gaussian measure and computed the \emph{mean} nodal volume. Consistently with Yau's conjecture, it turned out to be proportional to $\sqrt{E}$ by a constant factor (the volume of the manifold times an explicit constant that only depends on the dimension of the space). 
Since then, several authors investigated geometric properties of random eigenfunctions, motivated also by Berry's ansatz \cite{Be:77} on universality of high-energy eigenfunctions for ``generic'' classically chaotic billiards. 

In 2008 Rudnick and Wigman \cite{RW:08}, inspired by B\'erard's model, considered random eigenfunctions on the standard two-dimensional flat torus $\mathbb T^2:=\mathbb R^2/\mathbb Z^2$, the so-called \emph{arithmetic} random waves. In this case the spectral degeneracy properties are related to the ``sums of squares'' problem, see e.g. \cite{BB:14}. Indeed, the eigenvalues (or energy levels)
of the Laplace-Beltrami operator of $\mathbb T^2$ are all numbers $E=E_n:=4\pi^2n$ where $n$ is a sum of two squares, namely 
\begin{equation}\label{squares}
n=a^2 + b^2,\quad a,b\in \mathbb Z,
\end{equation} 
and the multiplicity of $E_n$, say $\mathcal N_n$, coincides with the number of pairs  $(a,b)\in \mathbb Z^2$ lying on the circle of radius $\sqrt{n}$. Landau's Theorem \cite{La:08} ensures that $\mathcal N_n$ grows on average as $\sqrt{\log n}$, however it could be as small as $8$ or as big as a power of $\log n$, depending on the chosen subsequence of energy levels. 

In order to go deeply into the erratic behavior of lattice points, we define the atomic probability measure $\mu_n$ on the unit circle by projecting points with coordinates $(a,b)\in \mathbb Z^2$ satisfying \eqref{squares}, attaching to them a Dirac mass and then averaging. It is well known that there exists a density-$1$ sequence of eigenvalues for which $\lbrace \mu_n\rbrace_n$ converges to the uniform probability measure on the unit circle. (In this case, the corresponding \emph{pointwise} scaling limit of the covariance kernel of arithmetic random waves is indeed those of Berry's random wave model \cite{Be:77, Be:02}, i.e., the Bessel function of the first kind of order zero.)
However, there are other weak-* limits classified in \cite{KW:17}. In particular, for every $\eta\in [-1,1]$ there exists a subsequence of eigenvalues whose corresponding sequence of probability measures $\lbrace \mu_n\rbrace_n$ is such that 
\begin{equation}\label{4fourier}
\widehat{\mu_n}(4) \to \eta,\qquad \textnormal{ as } \mathcal N_n\sto +\infty,
\end{equation}
where $\widehat{\mu_n}(4)$ denotes the $4$-th Fourier coefficient of $\mu_n$. Note that when the scaling limit is Berry's model, then $\eta=0$.

For the nodal lines of the $n$-th arithmetic random wave, Rudnick and Wigman proved the length $\Lc_n$ to be on \emph{average} as big as $\sqrt{E_n}$, the square root of the corresponding eigenvalue (analogously to the spherical case), in accordance with Yau's conjecture. More precisely,
\begin{equation}\label{media}
    \mathbb E[\mathcal L_n] = \frac{1}{2\sqrt 2} \sqrt{E_n}. 
\end{equation}
The sharp variance, for large eigenvalues (actually as $\mathcal N_n\sto +\infty$) has been found by Khrishnapur, Kurlberg and Wigman in 2013 \cite{KKW:13} to be
\begin{equation}\label{varianza}
\Var(\mathcal L_n) \sim \frac{1+\widehat{\mu_n}(4)^2}{512} \frac{E_n}{\mathcal N_n^2}.
\end{equation}
Remarkably, its behavior turns out to be non-universal, depending on fine properties of lattice points.
In particular, \eqref{media} and \eqref{varianza} together with Markov inequality entail concentration of the nodal length around its mean, giving some information on the constants in \eqref{yau}: as $\mathcal N_n\sto +\infty$,
$
    \mathcal L_n/\sqrt{E_n}\mathop{\to}^{\mathbb P} 1/2\sqrt 2$,
see Section \ref{sec-KKW} for more details. 

Regarding second order fluctuations of the nodal length of arithmetic random waves, Marinucci, Peccati, Rossi and Wigman \cite{MPRW:16} proved a non-universal non-central limit theorem: as $\mathcal N_n\sto +\infty$ s.t. \eqref{4fourier} holds, the limiting distribution of $\mathcal L_n$ is a linear combination of independent chi-square distributed random variables $X_1^2$ and $X_2^2$, whose coefficients depend on \eqref{4fourier}
\begin{equation}\label{conv}
  \frac{\mathcal L_n-\mathbb E[\mathcal L_n]}{\sqrt{\Var(\mathcal L_n)}} \mathop{\to}^d \frac{1}{2\sqrt{1+\eta^2}} (2- (1+\eta)X_1^2 -(1-\eta)X_2^2)=:\mathcal M_\eta,  
\end{equation}
see \eqref{e:r} for a complete discussion. 

In 2020, Benatar, Marinucci and Wigman \cite{BMW:20} studied the nodal length $\Lc_n$ of high-energy arithmetic random waves restricted to \emph{shrinking domains}, all the way down to Planck scale. In particular, they consider centered balls $B(s_n)$ with vanishing radius $s_n$ such that $s_n \sqrt{E_n}\sto +\infty$, see \cite{BMW:20} for precise assumptions. (Indeed, when valid, Berry's random wave model is applicable to shrinking domains of radius $\approx \frac{C}{\sqrt{E}}$ with $C\gg 0$, see the Introduction in \cite{BMW:20}.) 
In this work they found that, up to a natural scaling, the variance of the restricted nodal length $\mathcal L_{n;s_n}$ obeys the same asymptotic law as the total nodal length $\mathcal L_n$, and remarkably global and local nodal lengths are asymptotically fully correlated. In particular this implies $\mathcal L_{n;s_n}$ to exhibit the same limiting behavior as $\mathcal L_n$ in \eqref{conv}.

\ 

The present paper is installed within these results as a refinement of the non-central limit theorems for both total and local nodal length of arithmetic random waves established respectively in \cite{MPRW:16} and \cite{BMW:20}.
We are inspired by \cite{MRT:21} where the authors investigate the asymptotic behavior of the nodal length for random Laplace eigenfunctions on the unit two-dimensional round sphere, see Section \ref{sec-MR} for a complete discussion. 

In the toral case, we are able to prove a non-universal non-central Moderate Deviation Principle (MDP) -- see Section \ref{sub:MDPs} for the definition -- collected in Theorems \ref{MR} and \ref{MR-shrinking} with an explicit rate function depending on \eqref{4fourier}. In particular, we quantify at a logarithmic scale the asymptotic probability of \emph{rare} events such as $
\mathbb P \Big ((\mathcal L_n-\mathbb E[\mathcal L_n])/\sqrt{\Var(\mathcal L_n)} \le -y\cdot \alpha_n \Big )$,
where $y > 0$ and $\lbrace \alpha_n\rbrace_n$ is a sequence of positive numbers going to infinity slowly enough, in accordance with the growth of $\mathcal N_n$.
Our analysis provides further information on the nodal length of Laplacian eigenfunctions on the torus; indeed, our Theorem \ref{MR} implies in particular 
\begin{equation}\label{est1}
\mathbb P \Big (\mathcal L_n \le \mathbb E[\mathcal L_n] - y \cdot\alpha_n\sqrt{\Var(\mathcal L_n)}   \Big ) = e^{ -\alpha_n \Big(y  \frac{\sqrt{1 + \eta^2}}{1+|\eta|} + o(1)\Big )},
\end{equation}
as $\mathcal N_n\sto+\infty$, where mean and variance are as in \eqref{media} and \eqref{varianza} respectively. 
A result similar to \eqref{est1} holds for the local nodal length by applying Theorem \ref{MR-shrinking}. For future research, it would be interesting to prove \emph{exponential} concentration of $\mathcal L_n$ around its mean, namely a Large Deviation principle for the nodal length of random eigenfunctions at least in the toral and spherical case.

\paragraph{Plan of the paper.} 
Section \ref{sec-BN} is devoted to background and notation needed in order to understand the  model of interest and the goal of the this paper: we first define arithmetic random waves and present prior results on their nodal length, both on the whole torus and in geodesic balls, and then we recall basic notions in Large Deviation theory. In Section \ref{sec-MR} we formulate our main results, i.e., Theorems \ref{MR} and \ref{MR-shrinking}, together with a brief outline of their proofs which allows us to make a comparison with the spherical case. The last sections are devoted to the detailed proofs: our argument relies on the interplay between the long memory behavior of arithmetic random waves and the chaotic expansion of the nodal length (Section \ref{secWiener}), as well as on well-known techniques in Large Deviation theory such as the contraction principle and the concept of exponential equivalence (Section \ref{sec-proofs}).

\section{Background and notation}\label{sec-BN}

In this section we introduce our random model and recall some
basic notions in Large Deviation theory, eventually describing our main results. 

\emph{Some conventions.} Given two sequences of positive numbers $\lbrace a_n\rbrace_n$ and $\lbrace b_n\rbrace_n$, we write $a_n \vee b_n$ to indicate the maximum between $a_n$ and $b_n$, $a_n=o(b_n)$ if $a_n/b_n \sto 0$ as $n\sto +\infty$, $a_n = O(b_n)$ or equivalently $a_n \ll b_n$  if $a_n/b_n$ is asymptotically bounded and $a_n \sim b_n$ if $a_n/b_n \sto 1$ as $n\sto +\infty$. Moreover, for random variables $\lbrace X_n\rbrace_n$, $X$ and $Y$ we write  $X \stackrel{d}{=} Y$ if $X$ and $Y$ share the same law, and finally $X_n \stackrel{d}{\to} X$ if the sequence $X_n$ converges to $X$ in distribution.

\subsection{Arithmetic random waves}

Let $\Tb^2:=\R^2/\Z^2$ be the standard
two-dimensional flat torus and $\Delta$ the Laplacian on $\Tb^2$. The eigenvalues $E\in \mathbb R$ 
of the Helmholtz equation
\begin{equation}
\label{eq:Schrodinger}
\Delta f + Ef=0
\end{equation}
are all numbers of the form $E_{n}=4\pi^{2}n$ with $n\in S$, where 
$$S=\{{{} n \in \Z : n} =  a^2+b^2 \,\,  \mbox{{} for some} \:a,b\in\Z\}.$$ 
In order to describe Laplace eigenspaces, for $n\in S$ we denote by
$\Lambda_n$ the set of frequencies
\begin{equation*}
\Lambda_n := \lbrace \lambda =(\lambda_1,\lambda_2)\in \Z^2 : \lambda_1^2 + \lambda_2^2 = n\rbrace\
\end{equation*}
whose cardinality $
\mathcal N_n := |\Lambda_n|
$
equals the number of ways to express $n$ as a sum of two square integers.
For $\lambda\in \Lambda_{n}$, we denote by 
\begin{equation*}
e_{\lambda}(x) := \exp(2\pi i \langle \lambda, x \rangle),\quad x=(x_{1},x_{2})\in\Tb^2,
\end{equation*}
the complex exponential associated to the frequency $\lambda$.
The family
$
\{e_{\lambda}\}_{\lambda\in \Lambda_n}
$
is an $L^{2}$-orthonormal basis of the eigenspace of $-\Delta$ corresponding to the
eigenvalue $E_{n}$. In particular, its dimension 
is
$
\mathcal N_n = |\Lambda_n|
$.
The number
$\Nc_{n}$ grows \cite{La:08} {\em on average}
as $\sqrt{\log{n}}$, but could be as small as $8$ for prime numbers
$p\equiv 1\mod{4}$, or as large as a power of $\log{n}$.

From now on, we assume that every random object considered in this paper is defined on a common complete probability space $(\Omega, \mathcal{F}, \P)$, with $\E$ denoting mathematical expectation with respect to $\P$.
\begin{definition} \label{defarw}
For $n\in S$, the $n$-th {\it arithmetic random wave} 
is the random field 
\begin{equation}\label{defT}
T_n(x):=\frac{1}{\sqrt{\mathcal N_n}}\sum_{ \lambda\in \Lambda_n}a_{\lambda}e_\lambda(x), \quad x\in \Tb^2,
\end{equation}
where the coefficients $a_{\lambda}$ are standard complex-Gaussian random variables verifying the following properties: $a_\lambda$ is stochastically independent of $a_\gamma$ whenever $\gamma \notin \{\lambda, -\lambda\}$, and
$a_{-\lambda}= \overline{a_{\lambda}}$ (ensuring that $T_{n}$ is real-valued).
\end{definition}
Let us define, for $n\in S$ such that $\sqrt{n}$ is not an integer,
\begin{equation}\label{Lambdaplus}
\Lambda_n^+ := \lbrace \lambda =(\lambda_1, \lambda_2) \in \Lambda_n : \lambda_2 >0 \rbrace,
\end{equation}
otherwise $
\Lambda_n^+ := \lbrace \lambda =(\lambda_1, \lambda_2) \in \Lambda_n : \lambda_2 >0 \rbrace \cup \lbrace (\sqrt{n}, 0)\rbrace
$.
We assume that 
$
\lbrace  a_\lambda\rbrace_{\lambda\in \Lambda^+_n, n\in S}
$ is a family of independent random variables, in particular $\lbrace T_n\rbrace_{n\in S}$ is a family of independent random fields.
From \eqref{defT}, $T_n$ is stationary, centered
Gaussian  with covariance function $r_n:\mathbb T^2\to [-1,1]$ 
\begin{equation*}
 \E[T_n(x) T_n(y)] = \frac{1}{\mathcal N_n}
\sum_{\lambda\in \Lambda_n}e_{\lambda}(x-y)=\frac{1}{\mathcal N_n}\sum_{\lambda\in \Lambda_n}\cos\left(2\pi\langle x-y,\lambda \rangle\right )=: r_{n}(x-y),\quad x,y\in\Tb^2.
\end{equation*}
Note that $r_{n}(0)=1$, i.e. $T_{n}$ has unit variance.

\subsection{Nodal length: prior work}

For $n\in S$, the nodal set $T_n^{-1}\lbrace 0\rbrace:=\lbrace x\in \mathbb T^2 : T_n(x)=0\rbrace$ is a.s. a smooth curve on the torus, 
we are interested in the {\it nodal length} of the random
eigenfunctions, i.e. the collection $\{\Lc_{n}\}_{n\in S}$ of all random variables of the form
\begin{equation}\label{e:length}
\Lc_n := \text{length}(T_n^{-1}\lbrace 0 \rbrace).
\end{equation}

\subsubsection{Mean and variance}\label{sec-KKW}

The authors of \cite{RW:08} computed the expected value of $\Lc_{n}$:
\begin{equation}\label{mediaL}
\E[\mathcal L_n]= \frac{1}{2\sqrt{2}}\sqrt{E_n},
\end{equation}
which is well reflecting Yau's conjecture \cite{Ya:82}.
A bound for the variance of $\Lc_{n}$ was obtained in  \cite{RW:08}, but the challenging task of attaining an exact asymptotic behaviour for $\var(\Lc_{n})$ was completely settled in \cite{KKW:13} as follows.
Let $\Sc^{1}\subset\R^{2}$ be the unit circle and set, for $n\in S$, $\mu_{n}$ to be the probability measure on $\Sc^{1}$ defined as
\begin{equation*}
\mu_{n} := \frac{1}{\mathcal N_n} \sum_{\lambda\in \Lambda_n} \delta_{\frac{\lambda}{\sqrt{n}}}.
\end{equation*}
Thanks to \cite{EH:99}, we know that there exists a density-$1$ subsequence $\{n_{j}\}_j\subseteq S$ such that
\begin{equation}
\label{eq:mun equidist}
\mu_{n_{j}}\Rightarrow \frac{d\phi}{2\pi}\,,
\end{equation}
namely $\mu_{n_{j}}$ converges to the uniform measure on the unit circle (where $\Rightarrow$ denotes weak-$*$ convergence of probability measures, and $d\phi$ is the Lebesgue measure on $\Sc^{1}$). Nevertheless, the sequence $\{\mu_{n}\}_{n\in S}$
has an infinity of other weak-$*$ adherent points, see \cite{Ci:93,KKW:13}, which are the so-called {\it attainable measures}, see \cite{KW:17, Sar18}. 

Finally, we can state the main result of \cite{KKW:13}, which  is the following: as $\mathcal N_n\sto +\infty$,
\begin{equation}
\label{eq:var leading KKW}
\var(\Lc_{n}) =c_n \frac{E_n}{\Nc_{n}^2}\left (1 + O\left ( \frac{1}{\mathcal N_n^{1/2}}\right ) \right ),\quad c_n = c(\mu_n) := \frac{1+\widehat{\mu}_n(4)^2}{512},
\end{equation}
where $\widehat \mu(k) = \int_{\Sc^{1}} z^{-k}\,d\mu(z)$, $k\in \mathbb{Z}$,
represent the Fourier coefficients of a measure $\mu$ on the unit circle.
Moreover, from \eqref{eq:var leading KKW} it is clear that, in order for $\Lc_{n}$ to exhibit an
asymptotic law 
one has to consider a subsequence $\{ n_{j} \}_j\subseteq S$ such that the limit
$\lim_{j\rightarrow\infty}| \widehat{\mu}_{n_{j}}(4)|$ exists. Indeed,
if $\{ n_{j}\}_j \subseteq S$  is a subsequence such that $\Nc_{n_{j}}\sto\infty$ and $\mu_{n_{j}}\Rightarrow\mu$ for some probability measure $\mu$ on $\Sc^{1}$, then  
\begin{equation}\label{eq:var leading KKW2}
\var(\Lc_{n_j})\sim c({ \mu}) \frac{E_{n_j}}{\Nc_{n_j}^2}, \quad c(\mu) := \frac{1+\widehat{\mu}(4)^2}{512}.
\end{equation} 
Thanks to \cite{KKW:13,KW:17}, we know that for every $\eta \in [-1,1]$
there exists a subsequence $\{ n_{j}\}_j \subseteq S$ such that, as $j\sto\infty$, $\Nc_{n_{j}}\sto\infty$ and
$$
\widehat{\mu}_{n_j}(4)\to\eta\,.
$$
As a consequence, the possible values of the ``asymptotic"
constant $c(\mu)$ cover the whole interval
$\left[\frac{1}{512},\frac{1}{256}\right]$.
In particular, for the full density subsequence $\{n_{j}\}_{j}\subseteq S$ such that the lattice points in $\Lambda_{n_{j}}$ are asymptotically equidistributed \eqref{eq:mun equidist}, $\widehat{\mu}_{n_j}(4) \rightarrow 0$. 
On the  other hand, the work by \cite{Ci:93} has shown that there are 
\emph{thin} (i.e., with density equal to zero) sequences $\{n_{j}\}_{j}\subseteq S$, with $\Nc_{n_{j}}\sto\infty$, such that $\mu_{n_j}$ converges weakly to an atomic probability measure supported at the four symmetric points $\pm1$, $\pm i$; hence, $\widehat{\mu}_{n_{j}}(4)^2\sto1$ and $c_{n_j}\sto 1/256$.


\subsubsection{Asymptotic distribution}\label{sec-MPRW}

From now on, for a (non-zero) finite variance random variable $X$, we set $\widetilde X$ to be its normalized version, that is
$$
\widetilde X:=\frac{X-\E[X]}{\sqrt{\Var(X)}}\,.
$$
Fix $\eta\in [-1,1]$  and let $\mathcal{M}_\eta$ be the random variable
\begin{equation}\label{e:r}
\mathcal{M}_\eta := \frac{1}{2\sqrt{1+\eta^2}} (2 - (1+\eta) X_1^2-(1-\eta) X_2^2),
\end{equation}
where $X_{1},X_{2}$ are independent standard Gaussians. Note that 
$$
\mathcal M_\eta \mathop{=}^{\rm d} \mathcal M_{-\eta},
$$ moreover $\mathcal{M}_\eta$ is not Gaussian, indeed its support is $\left (-\infty, 1/\sqrt{1+\eta^2} \right ]$. Except for $\eta_2 = - \eta_1$,  $\mathcal M_{\eta_1}$ and $\mathcal M_{\eta_2}$ have different laws if $\eta_1\ne \eta_2$.

\begin{theorem}[Theorem 1.1, \cite{MPRW:16}]
\label{thm:lim dist sep}

Let $\{ n_{j} \}_j\subseteq S$ be a subsequence of $S$ satisfying $\mathcal N_{n_{j}}\sto\infty$, such that the sequence
$\{\big|\widehat{\mu}_{n_j}(4)\big|\}_j$ converges, that is:
\begin{equation}\label{eta-thm}
|\widehat{\mu}_{n_j}(4)\big|\rightarrow \eta,
\end{equation}
 for some $\eta \in [0,1]$.
Then
\begin{equation}
\label{eq:Lctild->Meta}
\widetilde{\Lc}_{n_j}  \stackrel{\rm d}{\longrightarrow} \mathcal{M}_\eta,
\end{equation}
where $\mathcal{M}_\eta$ is defined as in \eqref{e:r}.
\end{theorem}

In \cite{PR:18} a quantitative version of Theorem \ref{thm:lim dist sep} has been proved.
Let us recall the definition of 1-Wasserstein distance (see $\text{e.g.}$ \cite[\S C]{NP:12} and the references therein): given two random variables $X, Y$ whose laws are $\mu_X$ and $\mu_Y$, respectively, the Wasserstein distance between $\mu_X$ and $\mu_Y$,
written $d_{\text{W}}(X,Y)$, is defined as
\begin{equation*}\label{d def gen2}
d_{\text{W}}\left(X, Y\right):= \inf_{(A,B)}\mathbb E\left[ \left | A-B \right |  \right ],
\end{equation*}
 where the infimum runs over all pairs of random variables $(A,B)$ with marginal laws $\mu_X $ and $\mu_Y$, respectively. We will mainly use the dual representation 
\begin{equation}\label{d def gen}
d_{\text{W}}\left(X, Y\right) = \sup_{h\in \mathcal H_1} \left |\mathbb E\left[h(X) - h(Y)\right]\right |,
\end{equation}
where $\mathcal H_1$ denotes the class of Lipschitz functions $h:\mathbb R\to \mathbb R$ whose Lipschitz constant is less or equal than $1$. Relation \eqref{d def gen} implies in particular that, if $d_{\rm W}(X_n, X)\to 0$, then 
$\displaystyle{X_n \mathop{\to}^d  X}$ (the converse implication is false in general).
\begin{theorem}[Theorem 2, \cite{PR:18}]
\label{mainth1}
Let $\{ n_{j} \}_j\subseteq S$ be a subsequence of $S$ satisfying $\mathcal N_{n_{j}}\rightarrow\infty$, such that the sequence
$\{\big|\widehat{\mu}_{n_j}(4)\big| \}_j$ converges, that is:
$$|\widehat{\mu}_{n_j}(4)\big|\rightarrow \eta,$$ for some $\eta \in [0,1]$.
Then
\begin{equation}\label{maineq}
d_{\rm W}\left(\widetilde {\mathcal L}_{n_j}, \mathcal M_\eta \right) \ll  \mathcal N_{n_j}^{-1/4}\,   \vee \,  \left | \left |\widehat{\mu}_{n_j}(4) \right | - \eta \right |^{1/2}.
\end{equation}
\end{theorem}

\subsubsection{Shrinking domains}

For $0<s<1/2$, we set
$$\Lc_{n;s} := \text{length}\left( T_{n}^{-1}\lbrace 0\rbrace \cap B(s) \right)$$ to be the nodal length of $T_{n}$ restricted to a radius-$s$ ball $B(s)$,
where by the stationarity of $T_{n}$ we may assume that $B(s)$ is centered. Kac-Rice formula (see e.g. \cite[Chapter 11]{AT:07}) immediately gives
\begin{equation*}
\E[\Lc_{n;s}] = \frac{1}{2\sqrt{2}} (\pi s^{2})\cdot \sqrt{E_{n}}.
\end{equation*}
One of the main results of \cite{BMW:20} is that the variance of the nodal length $\Lc_{n;s}$ of $T_{n}$ restricted to balls that are shrinking slightly above
the so-called Planck scale has a similar form \eqref{eq:var leading KKW}, and
in particular one has the precise identity \cite[(3.36)]{BMW:20}
$$
\Cov(\Lc_{n;s}, \Lc_n) = (\pi s^2) \cdot \Var(\Lc_n).
$$
In what follows we keep the notation of \cite{BMW:20} to avoid confusion.
\begin{theorem}[Theorem 1.1, \cite{BMW:20}]
\label{thm:toral univ gen}
For every $\eps>0$ there exists a density-$1$ sequence of numbers $$S'=S'(\eps)\subseteq S$$ so that the following hold.

\begin{enumerate}

\item Along $n\in S'$ we have $\Nc_{n}\rightarrow \infty$, and the set of accumulation
points of $\{\widehat{\mu}_{n}(4)\}_{n\in S'}$ contains the interval $[0,1]$.

\item For $n\in S'$, uniformly for all $s>n^{-1/2+\eps}$ we have
\begin{equation}
\label{eq:Var asympt gen}
\var( \Lc_{n;s}) = c_{n} \cdot (\pi s^{2})^{2}\cdot \frac{E_{n}}{\Nc_{n}^{2}}
\left(1+O_{\eps}\left(\frac{1}{\Nc_{n}^{1/2}} \right) \right),
\end{equation}
where $c_n$ is defined as in (\ref{eq:var leading KKW}),
and the constant involved in the `$O$'-notation depends on $\eps$ only.

\item
For random variables $X,Y$ we denote as usual their correlation $$\Corr(X,Y) := \frac{\Cov(X,Y)}{\sqrt{\var(X)}\cdot \sqrt{\var(Y)}}.$$
Then for every $\eps>0$ we have that
\begin{equation}
\label{eq:restr nod len full corr full}
\sup\limits_{s>n^{-1/2+\eps}}\left| \Corr(\Lc_{n;s},\Lc_{n})-1\right| \rightarrow 0,
\end{equation}
i.e. the nodal length $\Lc_{n;s}$ of $T_{n}$ restricted to a small ball is asymptotically fully correlated with the total nodal length
$\Lc_{n}$ of $T_{n}$, uniformly for all $s>n^{-1/2+\eps}$.
\end{enumerate}
\end{theorem}
Under the same assumptions as Theorem \ref{thm:toral univ gen}, in view of (\ref{eq:restr nod len full corr full}), $\Lc_{n;s}$ obeys the same limiting law (\ref{eq:Lctild->Meta}) as the total nodal length. 
The aim of this paper is to refine the result of Theorem \ref{thm:lim dist sep} and the consequences of Theorem \ref{thm:toral univ gen} by means of Large Deviation theory.

\subsection{Large Deviation Principles}\label{sub:MDPs}

Let $(\Omega, \mathcal F, \mathbb P)$ be a complete probability space,
and $\lbrace X_n\rbrace_{n\in \mathbb N}$ a sequence of random variables taking 
values in some topological space $\mathcal X$. For our purpose we can restrict ourselves to 
the case where $\mathcal X$ is a metric space; we denote by $d$ its metric 
and by $\mathcal B(\mathcal X)$ its Borel $\sigma$-field. 

\begin{definition}\label{defLDP}  We say that $\lbrace X_n\rbrace_{n\in \mathbb N}$ satisfies the Large Deviation Principle (LDP) with speed $0\le s_n\nearrow +\infty$ and (good) rate function $\mathcal I:\mathcal X\to [0,+\infty]$ if the level sets $\lbrace x : \mathcal I(x) \le \alpha \rbrace,\alpha\geq 0$ are compact and for all $B\in \mathcal B(\mathcal X)$ we have 
\begin{equation*}
- \inf_{x\in \mathring{B}} \mathcal I(x)\le \liminf_{n\sto +\infty} \frac{1}{s_n} \log \mathbb P(X_n\in B) \le \limsup_{n\sto +\infty} \frac{1}{s_n} \log \mathbb P(X_n\in B) \le - \inf_{x\in \bar{B}} \mathcal I(x),
\end{equation*}
where $\mathring B$ (resp. $\bar B$) denotes the interior (resp. the closure) of $B$. 
\end{definition}
\begin{remark}\label{rem:LDP-consequence}\rm
In this paper, and in several common cases in the literature, there
exists
$x_0$ such that $\mathcal I(x)=0$ if and only if $x=x_0$. As a consequence, if $X_n$ satisfies the LDP in Definition \ref{defLDP}, then it
converges to $x_0$ as $n\sto +\infty$, at least in probability. Roughly speaking, one can
also say that
for every Borel set $B$ such that $x_0\notin\bar{B}$, the quantity $\mathbb P(X_n\in B)$
decays as $e^{-s_n \mathcal I(B)}$, where $\mathcal I(B):=\inf\{\mathcal
I(x):x\in B\}>0$.
Thus, in some sense, the larger is the rate function
$\mathcal I$
locally around $x_0$, the faster is the convergence to $x_0$.
\end{remark}

Let us state the G\"artner-Ellis Theorem (see point (c) of \cite[Theorem 2.3.6]{DZ:98}) in $\Rd$; in particular we denote the inner product
in $\Rd$ by $\langle\cdot,\cdot\rangle$. This theorem will be used in Section \ref{step1} with $d=3$.

\begin{theorem}\label{GET}
Let $\lbrace X_n\rbrace_{n\in \mathbb N}$ be a family of $\mathbb{R}^d$-valued random variables. Assume that $0\le s_n\nearrow +\infty$ and, for all 
$\theta\in\mathbb{R}^d$, the limit
$$\psi(\theta):=\lim_{n\sto +\infty}\frac{1}{s_n}\log\mathbb{ E}[e^{s_n \langle\theta, X_n\rangle}]$$
exists as an extended real number. Further, assume that the origin belongs to the interior $\mathcal{D}(\psi):=\{\theta\in\mathbb{R}^d:
\psi(\theta)<+\infty\}$. Then, if the function $\psi$ is essentially smooth
and lower 
semicontinuous, $\lbrace X_n\rbrace_{n\in \mathbb N}$ satisfies the LDP with speed $s_n$ and good rate function $\psi^*$ defined by
$\psi^*(x):=\sup_{\theta\in\R^d}\{\langle\theta, x\rangle-\psi(\theta)\}$ (that is the \emph{Legendre transform} of $\psi$).
\end{theorem}

Let us state the so-called \emph{contraction principle} \cite[Theorem 4.2.1]{DZ:98}.
\begin{theorem}\label{contractionprinciple}
Let $(\mathcal Y, \mathcal B(\mathcal Y))$ be a metric space with its Borel $\sigma$-field, and $f:\mathcal X \to \mathcal Y$ a continuous function. If $\lbrace X_n\rbrace_{n\in \mathbb N}$ satisfies a LDP with (good) rate function $\mathcal I$, then $\lbrace Y_n:= f(X_n)\rbrace_{n\in \mathbb N}$ satisfies a LDP with (good) rate function $\mathcal I_f:\mathcal Y\to [0,+\infty]$ defined as 
\begin{equation*}
\mathcal I_f(y) := \inf_{x\in f^{-1}(y)} \mathcal I(x).
\end{equation*}
\end{theorem}
Theorem \ref{contractionprinciple} ensures that the LDP is preserved under continuous transformations. In order to extend the contraction principle beyond the continuous case, it is beneficial to recall the notion of exponential equivalence \cite[Definition 4.2.10]{DZ:98}.
\begin{definition}\label{expdef} Let $\lbrace X_n\rbrace_{n\in \mathbb N}$ and $\lbrace Y_n\rbrace_{n\in \mathbb N}$ be two sequences of random variables taking values in $\mathcal X$. We say that they are exponentially equivalent at speed $0\le s_n\nearrow +\infty$ if, for every $\delta>0$, the set $\Gamma_\delta:=\lbrace d(X_n, Y_n) > \delta\rbrace\subseteq \Omega$ is measurable and 
\begin{equation}\label{exp_equiv}
\limsup_{n\sto +\infty} \frac{1}{s_n} \log \mathbb P (\Gamma_\delta ) = -\infty.
\end{equation} 
\end{definition}
The following theorem states that, from the point of view of Large Deviations, sequences of random variables that are exponentially equivalent are identical, see also \cite[Theorem 4.2.13]{DZ:98}. 
\begin{theorem}\label{thDZ}
Assume that $\lbrace X_n\rbrace_{n\in \mathbb N}$ satisfies the LDP with speed $s_n$ and good rate function $\mathcal I$.
Then, if $\lbrace X_n\rbrace_{n\in \mathbb N}$ and $\lbrace Y_n\rbrace_{n\in \mathbb N}$ are exponentially equivalent at
speed $s_n$, the same LDP holds for $\lbrace Y_n\rbrace_{n\in \mathbb N}$.
\end{theorem}

A Moderate Deviation Principle (MDP) is a class of LDPs that allows to
\emph{fill the
gap} between the following asymptotic regimes:
\begin{itemize}
\item a convergence to a constant, which is governed by a suitable LDP;
\item a weak convergence to a centered Normal distribution.
\end{itemize}
Some examples of classes of LDPs of this kind in which the weak convergence is towards a non-Gaussian law can be found e.g. in \cite{GM:21} (see also the references cited therein). It is worth stressing that actually this kind of phenomena are quite natural for random sums (in particular with some suitable compound Geometric distributions).


Typically a MDP is class of LDPs concerning families of random variables
which depend on
the choice of certain scalings satisfying some suitable conditions,
and these LDPs
are governed by the same rate function. As a prototype example we can
consider the
empirical means of i.i.d. random variables, see e.g. \cite[Theorem
3.7.1]{DZ:98}, where
the asymptotic regimes cited above concern the Law of Large Numbers
(LLN) and the Central
Limit Theorem (CLT); in particular a LDP related to the LLN is provided
by the well-known
Cram\'er Theorem \cite[Theorem 2.2.3]{DZ:98}.

Actually we cannot rigorously say that the MDP in this paper, see Theorems \ref{MR} and \ref{MR-shrinking}, (as well as the MDP in
\cite{MRT:21}) allows
to fill the gap between two asymptotic regimes as above because
for the moment we do not have
a LDP for the convergence to a constant for our statistics of the field. In the case of the nodal length, as anticipated in the Introduction, from \eqref{mediaL} and \eqref{eq:var leading KKW} we know that for every $\epsilon >0$ 
$$
\mathbb P \left ( \left |\frac{\mathcal L_n}{\sqrt{E_n}}- \frac{1}{2\sqrt 2}\right | > \epsilon \right ) =O_\epsilon \left (  \frac{1}{\mathcal N_n^2}\right ),\qquad \textnormal{as }\mathcal N_n\to +\infty$$
(where the constant involved in the $O$-notation depends on $\epsilon$), 
hence it is natural to explore exponential concentration for $\mathcal L_n/\sqrt{E_n}$ around its mean. 
Nevertheless we believe that the investigation of a LDP for the sequence of r.v.'s $\lbrace \mathcal L_n/\sqrt{E_n}\rbrace_n$ deserves a separate study, eventually together with a LDP for the nodal length of random spherical harmonics. However,
if it is possible to fill that gap, we would have a \emph{non-central}
MDP where the weak
convergence is towards the distribution of the random variable
$\mathcal{M}_\eta$ defined
in \eqref{e:r}. Since the limiting distribution of the nodal length is not universal, in particular it depends on the subsequence of energy levels, in accordance with the non-universality of our MDP stated above we say that our MDP is also \emph{non-universal}.
 
\section{Main Results}\label{sec-MR}

In this section we finally state our main results, which are non-universal non-central MDPs (see Section \ref{sub:MDPs}) refining Theorem \ref{thm:lim dist sep} and the consequences of Theorem \ref{thm:toral univ gen}, specifically a class of LDPs for the nodal length of Arithmetic Random Waves for certain scalings $\{\alpha_{n_j}\}_j$ ($\{n_j\}_j \subset S$) that grow to infinity slowly as $j\sto \infty$ (see condition \eqref{cond:scaling}), with speed $\alpha_{n_j}$, and  an $\eta$-dependent rate function $I_\eta$ (see (\ref{eta-thm})). Moreover, for $\alpha_{n_j}\equiv1$ (note that in this case the condition \eqref{cond:scaling} fails) we have the convergence in distribution to the random variable $\mathcal M_\eta$  defined in \eqref{e:r}, that is, we retrieve Theorems \ref{thm:lim dist sep} and the consequences of Theorem \ref{thm:toral univ gen}.

\begin{theorem}\label{MR}
Let $\{n_j\}_j\subseteq S$ be such that $\mathcal N_{n_j}\sto+\infty$ and $\widehat{\mu}_{n_j}(4)\sto\eta\in[-1,1]$ as $j\sto+\infty$, and $\{\alpha_{n_j}\}_j$ be any sequence of positive numbers such that, as $j\sto+\infty$,
\begin{equation}\label{cond:scaling}
\alpha_{n_j} \to +\infty \qquad \text{and} \qquad \alpha_{n_j}/\log \mathcal N_{n_j} \to 0\,.
\end{equation}
Then the sequence of random variables 
$
\{\widetilde\Lc_{n_j}/\alpha_{n_j}\}_j
$
satisfies a Moderate Deviation Principle (MDP) with speed $\alpha_{n_j}$ and rate function 
\begin{equation}\label{rate-function}
I_\eta(y)=\begin{cases}
\ds -y\frac{\sqrt{1 + \eta^2}}{1+|\eta|}&\text{ if } \,\, y\le 0 \\[12pt]
\ds  +\infty &\text{ if }\,\, y>0
\end{cases}\,.
\end{equation}
\end{theorem}
\begin{theorem}\label{MR-shrinking}
Let $\{n_j\}_j\subseteq S$ and ${n_j}^{-\frac12+\eps}<s_{n_j}<1/2$ be such that $s_{n_j}\sto0$. If $\{\alpha_{n_j}, \,j\in\N\}$ is any sequence of positive numbers such that, as $j\sto+\infty$,
$$
\alpha_{n_j} \to +\infty \qquad \text{and} \qquad \alpha_{n_j}/\log \mathcal N_{n_j} \to 0\,,
$$ 
then the sequence of random variables 
$
\{\widetilde\Lc_{n_j,{s_{n_j}}}/\alpha_{n_j}\}_j$
satisfies a Moderate Deviation Principle (MDP) with speed $\alpha_{n_j}$ and the same rate function as in \eqref{rate-function}.
\end{theorem}

Clearly, the sequences of random variables in Theorems \ref{MR}
and
\ref{MR-shrinking} converge to zero (this is a known consequence of the
weak convergence of $\widetilde \Lc_{n_j}$ and $\widetilde \Lc_{n_j,s}$, combined with the
Slutsky Theorem); indeed their common rate function uniquely
vanishes at $y=0$, see also Remark \ref{rem:LDP-consequence}. Then, in both cases, the smaller is $|\eta|$, the
larger
is the rate function $I_\eta$ in (\ref{rate-function}). Then, considering once more
what
we said in Remark \ref{rem:LDP-consequence}, the smaller is $|\eta|$,
the
faster is the convergence to zero.

\begin{remark}\rm
It is worth noticing that the condition on the speed $\alpha_n$ in (\ref{cond:scaling}) may be suboptimal. However, our technique is flexible enough to deal
with other geometric properties of nodal sets of arithmetic random waves, even in higher dimension \cite{Maf:19, Ca:19}. All these interesting cases were left uncovered by the strategy pursued in \cite{MRT:21}, where the spherical counterpart has been investigated and (central) MDPs 
have been established for the nodal length of random spherical harmonics (Gaussian Laplace eigenfunctions on the round sphere) both on the whole manifold and on shrinking spherical caps. It is worth stressing that on the sphere these objects have Gaussian fluctuations in the high-energy limit, and are independent, in marked contrast with the toral case where they show a non-universal non-central asymptotic behavior, and full correlation. See Section \ref{ontheproof} for more details on this comparison. 
\end{remark} 

\subsection{On the proof of the main results}\label{ontheproof}

In order to list and motivate the proof steps, we start by recalling that our nodal length $\mathcal L_n$ in (\ref{e:length}) lives in the Wiener chaos (being a finite variance functional of a Gaussian field), in particular it can be written as an orthogonal series, converging in $L^2(\mathbb P)$, of the form
\begin{equation}\label{serie}
\mathcal L_n = \mathbb E[\mathcal L_n] + \sum_{q=1}^{+\infty} \mathcal L_n[q],
\end{equation}
where $\mathcal L_n[q]$ denotes the orthogonal projection of $\mathcal L_n$ into the so-called $q$-th Wiener chaos. Roughly, this expansion relies on the fact that Hermite polynomials form an orthogonal basis of the space of square integrable functions on the real line with respect to the Gaussian density. It turns out that projections onto \emph{odd} order chaoses vanish, and moreover $\mathcal L_n[2]=0$ (related to so-called Berry's cancellation phenomenon \cite{KKW:13, MPRW:16}).

Our argument relies on the fact that the asymptotic behavior of the (centered) nodal length $\Lc_n-\mathbb E[\mathcal L_n]$ is  completely determined by $\mathcal L_n[4]$, its fourth chaotic component \cite{MPRW:16}, and, moreover, that the dominant term of $\widetilde \Lc_n[4]$ (where $\widetilde \Lc_n[4]$ denotes the standardized fourth chaos of the nodal length), say $\mathcal M_{n}=\mathcal M_{\widehat \mu_n(4)}$ \emph{abusing notation}, is a continuous function $f:\mathbb R^3 \to \mathbb R$ 
of a random vector $\widetilde S_n$ (whose components live in the $2$-nd Wiener chaos), 
$\widetilde S_n$ converging in law towards a multivariate Gaussian. First we establish a Moderate Deviation Principle for $\widetilde S_n/\sqrt{\alpha_n}$ at speed $\alpha_n$, and then, thanks to the celebrated contraction principle (see \cite[Theorem 4.2.1]{DZ:98} as well as Theorem \ref{contractionprinciple}), we transfer this result to $\mathcal M_n/\alpha_n$ with an explicit rate function (indeed, $f$ is a multivariate polynomial) and speed $\alpha_n$. It remains to deal with the tail of the series (\ref{serie}), i.e. with
\begin{equation*}
\widetilde{\mathcal L}_n  - \mathcal M_n.
\end{equation*} 
To this aim, we prove that $\widetilde{\mathcal L}_n/\alpha_n$ and $\mathcal M_n/\alpha_n$ -- under some additional constraint on the speed $\alpha_n$ -- are exponentially equivalent (Definition \ref{expdef}), so that they satisfy the same Deviation Principle (Theorem \ref{thDZ}). 
Finally, we will take advantage of the full correlation result in \cite[Theorem 1.5]{BMW:20} (see also Theorem \ref{thm:toral univ gen})
 to check that $\widetilde{\Lc}_n /\alpha_n$ and $\widetilde{\Lc}_{n;s}/\alpha_n$, the standardized nodal length restricted to the ball $B(s)$ whose radius is slightly above the Planck scale, are exponentially equivalent, thus sharing the same deviations.
 
\subsection{The spherical case}
 
Our proofs are inspired by \cite{MRT:21}, where a Moderate Deviation Principle has been established for the nodal length of \emph{random spherical harmonics}, both on the whole sphere and on shrinking spherical caps. However, there are marked differences with the toral case. Indeed, on the sphere the limiting distribution of these two geometric functionals is Gaussian, and they are asymptotically independent  (see \cite{To:20} and the references therein). On $\mathbb T^2$ instead, the high-energy behavior of the total nodal length is non-Gaussian and non-universal (Theorem \ref{thm:lim dist sep}), and it is asymptotically fully correlated with the nodal length in shrinking domains slightly above the Planck scale (Theorem \ref{thm:toral univ gen}).
Moreover, the rate function $J$ in the main results of \cite{MRT:21} is quadratic
\begin{equation*}
J(y) = y^2/2,\quad y\in \mathbb R,
\end{equation*}
while our rate function \eqref{rate-function} is a line whose angular coefficient depends on $\eta$. However, \emph{the condition (\ref{cond:scaling}) on the MDP speed is comparable in some sense to those on the sphere (see (2.1) and (2.2) in \cite{MRT:21}), indeed $\mathcal N_n$ grows on average as $\sqrt{\log n}$ \cite{La:08}.} 

For the nodal length on the sphere, the starting point is a chaotic expansion similar to (\ref{serie}) and also in this case the dominant term is the fourth chaotic component, that however behaves much differently than the corresponding term on the torus: it is equivalent in the high-energy limit to the so-called sample trispectrum (i.e., the integral of the fourth Hermite polynomial evaluated at the field), which is \emph{not} true in the toral case, and it is asymptotically Gaussian. To prove a MDP for the total nodal length on the sphere, the authors of \cite{MRT:21} first showed a MDP for the sample trispectrum via the cumulant approach presented in \cite{ST16}, which is a link between the fourth moment theorem for Gaussian approximation \cite{NP:12} and the LD theory, and then checked its exponential equivalence with the total nodal length (a similar argument works for shrinking domains).

\section{Chaotic expansions}\label{secWiener}

In this section we recall the chaotic expansion of the nodal length for arithmetic random waves, that is crucial for the proof of our main results. In particular, at the end of Section \ref{secWiener} we focus on the fourth chaotic component  of $\Lc_n$ and its dominant term $\mathcal M_n$. 

\subsection{Wiener chaos}

The family of Hermite polynomials $\lbrace H_q \rbrace_{q\in \mathbb N}$ is defined as follows: $H_0 \equiv 1$ and
 \begin{equation}\label{hermite}
 H_{q}(t) := (-1)^q \phi^{-1}(t) \frac{d^q}{dt^q} \phi(t), \qquad t\in \mathbb R, q\in \mathbb N_{\ge 1},
 \end{equation}
 where $\phi$ denotes the standard Gaussian density. 
It is well known \cite[Proposition 1.4.2]{NP:12} that $ \{H_q/\sqrt{q!} \}_{q\in \mathbb N}$ is a complete orthonormal system in the space of square integrable real functions with respect to the standard Gaussian measure on the real line.

Arithmetic Random Waves \eqref{defT} are generated (see Definition \ref{defarw}) from a family of standard complex-valued Gaussian random variables $\{a_\xi\}_{\xi\in \mathbb{Z}^2}$, defined on  $(\Omega, \mathscr{F}, \mathbb{P})$ and stochastically independent, save for the relations $a_{-\xi} = \overline{a_\xi}$.

Let $X$ be the closure in the Hilbert space $L^2(\mathbb{P})$ (with respect to the scalar product $\langle F, G\rangle := \mathbb E[(F-\mathbb E[F])(G - \mathbb E[G])], F,G\in L^2(\mathbb P))$ of all real finite linear combinations of random variables $\zeta$ of the form $$\zeta = z \, a_\xi + \overline{z} \, a_{-\xi}\,,$$ where $\xi\in \mathbb{Z}^2$ and $z\in \mathbb{C}$,
thus $X$ is a real centered Gaussian Hilbert subspace
of $L^2(\mathbb{P})$.
Now let $q\in \mathbb N$; the $q$-th Wiener chaos $C_q$ associated with $X$ is defined as the closure in $L^2(\mathbb{P})$ of all real finite linear combinations of random variables of the form
$$
H_{p_1}(\zeta_1) H_{p_2}(\zeta_2)\cdots H_{p_k}(\zeta_k)
$$
for $k\in \mathbb N_{\ge 1}$, where $p_1,...,p_k \in \mathbb N$ satisfy $p_1+\cdots+p_k = q$, and $(\zeta_1, \zeta_2, \dots, \zeta_k)$ is a standard Gaussian vector extracted
from $X$ ($C_0 = \mathbb{R}$ and $C_1 = X$). Note that, from (\ref{defT}), for every $n$ the random fields $T_n,\,\nabla T_n$ viewed 
as collections of Gaussian random variables indexed by $x\in \mathbb T^2$ are all lying in $X$.

It turns out that (see e.g. \cite[Theorem 2.2.4]{NP:12}) $C_q \,\bot\, C_{q'}$ in  $L^2(\mathbb{P})$ whenever $q\neq q'$, and moreover
\begin{equation*}
L^2_X(\mathbb{P}) = \bigoplus_{q=0}^\infty C_q,
\end{equation*}
where $L^2_X(\mathbb P) := L^2(\Omega, \sigma(X), \mathbb P)$, that is, every finite-variance real-valued functional $F$ of $X$ admits a unique representation as a series, converging in $L^2_X(\mathbb P)$, of the form
\begin{equation}\label{e:chaos2}
F = \sum_{q=0}^\infty F[q],
\end{equation}
$F[q]:=\text{proj}(F \, | \, C_q)$ being the orthogonal projection of $F$ onto $C_q$ (in particular, $F[0]= \E [F]$).
 For a complete discussion on Wiener chaos see \cite[\S 2.2]{NP:12} and the references therein. 

Since the nodal length $\Lc_n$ is a finite-variance \emph{explicit} functional (\ref{formalLength}) of the Gaussian field $T_n$, its Wiener-\^Ito chaos expansion (\ref{serie})
can be fruitfully exploited, as we will see in the proofs of our main results in Section \ref{sec-proofs}. 

\subsection{Explicit formulas}

The nodal length (\ref{e:length}) can be formally written as
\begin{equation}\label{formalLength}
\mathcal L_n = \int_{\mathbb T} \delta_0(T_n(\theta))\|\nabla T_n(\theta)\|\,d\theta,
\end{equation}
where $\delta_0$ denotes the Dirac delta function and $\|\cdot \|$ the
Euclidean norm in $\R^2$ (see \cite[Lemma 3.1]{RW:08}). Clearly, we mean that the approximating random variables $\mathcal L_n^\epsilon$ defined by
replacing the Dirac mass $\delta_0$ with $1_{[-\epsilon, \epsilon]}/(2\epsilon)$, for $\epsilon >0$, in (\ref{formalLength}), converge a.s. and in $L^2(\mathbb P)$ 
to $\text{length}(T_n^{-1}\lbrace 0 \rbrace)=\mathcal L_n$. 

Note that a straightforward differentiation of the definition \eqref{defT} of $T_n$ yields, for $j=1,2$
\begin{equation}\label{e:partial}
\partial_j T_n(x) = \frac{2\pi i}{\sqrt{\mathcal{N}_n} }\sum_{(\lambda_1,\lambda_2)\in \Lambda_n} \lambda_j a_\lambda e_\lambda(x),
\end{equation}
(here $\partial_j = \frac{\partial}{\partial x_j}$).
Hence the random fields $T_{n},\partial_{1} T_n,\partial_{2} T_n$ viewed as collections of
Gaussian random variables
indexed by $x\in\Tb^2$ are all lying in $X$, i.e. for every $x\in\Tb^2$ we have
\begin{equation*}
T_{n}(x),\, \partial_{1}T_{n}(x), \, \partial_{2}T_{n}(x) \in X.
\end{equation*}

We shall often use the following result from \cite{RW:08}:

\begin{lemma}[\cite{RW:08}, (4.1)]\label{lemmavar}
For $j=1,2$ we have that
$$\displaylines{
\Var[\partial_j T_n(x)] = \frac{4\pi^2}{\mathcal N_n}
\sum_{\lambda=(\lambda_1,\lambda_2)\in \Lambda_n} \lambda_j^2 = 4\pi^2 \frac{n}{2},
}$$
where the derivatives $\partial_j T_n(x)$ are as in \eqref{e:partial}.
\end{lemma}
Accordingly, for $x=(x_1, x_2)\in \mathbb T$ and $j=1,2$, we will denote by $\partial_j \widetilde T_n(x)$ the normalized derivative
\begin{equation}\label{e:norma}
\partial_j \widetilde T_n(x) := \frac{1}{2\pi} \sqrt{\frac{2}{n}} \frac{\partial}{\partial x_j}  T_n(x) = \sqrt{\frac{2}{n}}\frac{ i}{\sqrt{\mathcal N_n}}\sum_{ \lambda\in \Lambda_n}\lambda_j\,
a_{\lambda}e_\lambda(x).
\end{equation}
In view of convention \eqref{e:norma}, we formally rewrite \eqref{formalLength} as
\begin{equation*}\label{formalLength2}
\mathcal L_n = \sqrt{\frac{4\pi^2n}{2}}\int_{\mathbb T}
\delta_0(T_n(x)) \sqrt{(\partial_1 \widetilde
T_n(x))^2+(\partial_2 \widetilde T_n(x))^2}\,dx.
\end{equation*}
We also introduce two collections of coefficients
$\{\alpha_{2n,2m} : n,m\geq 1\}$ and $\{\beta_{2l} : l\geq 0\}$, that are related to the Hermite expansion of the norm $\| \cdot
\|$ in $\R^2$ and the (formal) Hermite expansion of the Dirac mass $ \delta_0(\cdot)$, respectively.
These are given by
\begin{equation}\label{e:beta}
\beta_{2l}:= \frac{1}{\sqrt{2\pi}}H_{2l}(0),
\end{equation}
where $H_{2l}$ denotes the $2l$-th Hermite polynomial, and
\begin{equation}\label{e:alpha}
\alpha_{2n,2m}=\sqrt{\frac{\pi}{2}}\frac{(2n)!(2m)!}{n!
m!}\frac{1}{2^{n+m}} p_{n+m}\left (\frac14 \right),
\end{equation}
where for $N\in \mathbb N$ and $x\in \R$
\begin{equation*}
\displaylines{ p_{N}(x) :=\sum_{j=0}^{N}(-1)^{j}\cdot(-1)^{N}{N
\choose j}\ \ \frac{(2j+1)!}{(j!)^2} x^j. }
\end{equation*}

\begin{proposition}[\cite{MPRW:16}, Proposition 3.2]
\label{teoexp}
\
{\rm (a)} For $q=2$ or $q=2m+1$ odd ($m\ge 0$),
\begin{equation*}\label{point1}
\Lc_n[q] \equiv 0,
\end{equation*}
that is, the corresponding chaotic projection vanishes.

{\rm (b)} For $q=0$ or $q\ge 2$
\begin{eqnarray}\label{e:pp}
\nonumber &&\Lc_n[2q]\\
&&= \sqrt{\frac{4\pi^2n}{2}}\sum_{u=0}^{q}\sum_{k=0}^{u}
\frac{\alpha _{2k,2u-2k}\beta _{2q-2u}
}{(2k)!(2u-2k)!(2q-2u)!} \times\\
&&\hspace{4cm}  \times \int_{\mathbb T}\!\! H_{2q-2u}(T_n(x))
H_{2k}(\partial_1 \widetilde T_n(x))H_{2u-2k}(\partial_2
\widetilde T_n(x))\,dx.\notag
\end{eqnarray}
Consolidating the above, the Wiener-It\^o chaotic expansion of $\Lc_n$ is
\begin{eqnarray*}
\Lc_n = \E [\Lc_n ]+ \sqrt{\frac{4\pi^2n}{2}}\sum_{q=2}^{+\infty}\sum_{u=0}^{q}\sum_{k=0}^{u}
\frac{\alpha _{2k,2u-2k}\beta _{2q-2u}
}{(2k)!(2u-2k)!(2q-2u)!}\times\\
\nonumber
\times \int_{\mathbb T}H_{2q-2u}(T_n(x))
H_{2k}( \widetilde \partial_1 T_n(x))H_{2u-2k}( \widetilde \partial_2
 T_n(x))\,dx,
\end{eqnarray*}
in $L^2(\P)$.
\end{proposition}

\subsubsection{The fourth chaotic component}

For $n\in S$, we set
\begin{eqnarray}
W(n):=\left(
\begin{array}{c}
W_{1}(n) \\
W_{2}(n) \\
W_{3}(n) \\
W_{4}(n)
\end{array}%
\right ) & :=& \frac{1}{n\sqrt{\Nc_{n}/2}}\sum_{\lambda = (\lambda_1,\lambda_2)\in \Lambda^+_{n}  }\left(
|a_{\lambda }|^{2}-1\right) \left(
\begin{array}{c}\label{Wdef}
n \\
\lambda _{1}^{2} \\
\lambda _{2}^{2} \\
\lambda _{1}\lambda _{2}%
\end{array}%
\right),
\end{eqnarray}
where $\Lambda_n^+$ is defined as in (\ref{Lambdaplus}).
Starting from the formula \eqref{e:pp} in the case $q=2$, the authors of \cite{MPRW:16} show the following.

\begin{lemma}[Lemma 4.2, \cite{MPRW:16}, Lemma 4, \cite{PR:18}]\label{p:formula4}
We have, for diverging subsequences $\lbrace n_j \rbrace \subseteq S$ such that $\mathcal N_{n_j}\sto +\infty$ and $\widehat \mu_{n_j}(4)$ converges,
\begin{equation*}
\mathcal{L}_{n_j}[4]=\sqrt{\frac{E_{n_j}}{ 512\, \mathcal N^2_{n_j}}}\Big(W_1(n_j) ^2-2W_2(n_j)^2-2W_3(n_j)^2 - 4W_4(n_j)^2+ R(n_j) \Big),
\end{equation*}
where 
$$
R(n_j)=\frac{1}{2\mathcal N_{n_j}}\sum_{\lambda \in \Lambda_{n_j}} \abs{a_\lambda}^4
$$ 
is a sequence of random variables converging in probability to $1$.
\end{lemma}
It is crucial to note that (see \cite[(46)]{PR:18})
$$
W_2(n)+W_3(n)=W_1(n)\,,
$$
which implies that the fourth chaotic component of ${\Lc}_{n_j}$ can be written as 
\begin{equation}\label{4chaosred}
\mathcal{L}_{n_j}[4]=\sqrt{\frac{E_{n_j}}{ 512\, \mathcal N^2_{n_j}}}\Big(\(W_2(n_j)+W_3(n_j)\)^2-2W_2(n_j)^2-2W_3(n_j)^2 - 4W_4(n_j)^2+ R(n_j) \Big)\,.
\end{equation}
We will use also the following important Lemma from \cite{MPRW:16}.
\begin{lemma}[Lemma 4.3, \cite{MPRW:16}]\label{p:clt} Assume that the subsequence $\{n_j\}_j\subseteq S$ is such that $\Nc_{n_j} \sto +\infty$ and $\widehat{\mu}_{n_j}(4) \sto \eta\in [-1,1]$.
Then, as $n_j\sto \infty$, the following CLT holds:
\begin{eqnarray}\label{e:punz}
W(n_j) \stackrel{\rm d}{\longrightarrow} Z(\eta)= \left(
\begin{array}{c}
Z_{1} \\
Z_{2} \\
Z_{3} \\
Z_{4}%
\end{array}%
\right ),
\end{eqnarray}
where $Z(\eta)$ is a centered Gaussian vector with covariance
\begin{equation}\label{e:sig}
\Gamma=\Gamma(\eta) =\left(
\begin{array}{cccc}
1 & \frac{1}{2} & \frac{1}{2} & 0 \\
\frac{1}{2} & \frac{3+\eta}{8}  & \frac{1-\eta}{8}  & 0 \\
\frac{1}{2} & \frac{1-\eta}{8}  & \frac{3+\eta}{8}  & 0 \\
0 & 0 & 0 & \frac{1-\eta}{8}
\end{array}%
\right).
\end{equation}%
The eigenvalues of $\Gamma$ are $ 0,\frac{3}{2},%
\frac{1-\eta}{8},\frac{1+\eta}{4}$, in particular $\Gamma$ is singular.
\end{lemma}
The fourth chaotic component $\Lc_n[4]$ is the dominating term in the series expansion (found in Proposition \ref{teoexp}) of the total nodal length $\Lc_n$. Indeed, in \cite{MPRW:16} it was proved that, as $\mathcal N_n\sto +\infty$,
\begin{equation}\label{lim var}
\text{Var}(\mathcal L_n)\sim \text{Var}(\mathcal L_n[4]),
\end{equation}
by showing that the asymptotic variance of $\mathcal L_n[4]$ equals the right hand side of \eqref{eq:var leading KKW}. Indeed, 
\begin{equation}\label{varesatta4chaos}
\text{Var}(\mathcal L_n[4]) = \frac{E_n}{512 \mathcal N_n^2} \left (1 + \widehat \mu_n(4)^2 + \frac{34}{\mathcal N_n} \right ).
\end{equation}
The asymptotic equality in \eqref{lim var} and the orthogonality properties of Wiener chaoses, that we have seen  in Section \ref{secWiener}, guarantee that, as $\mathcal N_n\sto +\infty$,
\begin{equation}\label{= law}
\frac{\mathcal L_n - \mathbb E[\mathcal L_n]}{\sqrt{\text{Var}(\mathcal L_n)}} = \frac{\mathcal L_n[4]}{\sqrt{\text{Var}(\mathcal L_n[4])}} + o_{\mathbb P}(1),
\end{equation}
where $o_{\mathbb P}(1)$ denotes a sequence converging to $0$ in probability. 

Moreover, thanks to \cite[Lemma 2]{PR:18}, which was proved using a powerful result by Bombieri and Bourgain on the ``sums of two squares'' problem \cite[Theorem 1]{BB:14}, we know that
\begin{equation}\label{PR-bound}
\E\[ \abs{\Lc_{n}-\Lc_{n}[4]}^2\]=O\left (\frac{E_n}{\mathcal  N_n^{5/2}}\right )\,.
\end{equation}

\section{Proofs}\label{sec-proofs}

In the present section we give the proofs of our main results.

\subsection{Proof of Theorem \ref{MR}}

From now on $\{n_j\}_j\subseteq S$ will denote a sequence of energy levels such that $\mathcal N_{n_j}\sto+\infty$ and $\widehat{\mu}_{n_j}(4)\sto\eta\in[-1,1]$ as $j\sto+\infty$: for the sake of notation brevity we will write $n$ instead of $n_j$ in the sequel. 

As anticipated in Section \ref{ontheproof}, the proof of Theorem \ref{MR} is divided into three steps. The first step is a MDP for a random vector, called $S_n$, whose components are linear combinations of independent and centered chi-square random variables, see Section \ref{step1}. Then we will show that our functional of interest, which is the nodal length $\Lc_n$ is \emph{exponentially equivalent} to a simpler functional that we will call $\mathcal M_n$, see Section \ref{step2}. In the third  and  final step Section \ref{step3}, we will prove through a contraction principle that $\mathcal M_n$ (and hence $\Lc_n$), which is a continuous function of $S_n$, satisfies a MDP with rate function as in \eqref{rate-function}.

\subsubsection{MDP for $S_n$}\label{step1}
Recalling the content of Section \ref{secWiener}, we define 
\begin{equation}\label{defSn}
S_n:=\sum_{\lambda\in\Lambda_{n}^+}\(\abs{a_\lambda}^2-1\) \( \begin{array}{c}
\lambda_1^2/n  \\
\lambda_2^2/n  \\
\lambda_1\lambda_2/n \end{array} \)\,;
\end{equation}
in particular $S_n$ is a linear combination of independent and centered chi-square random variables, where the coefficients are three-dimensional. 
From Lemma \ref{p:clt} which is \cite[Lemma 4.3]{MPRW:16}, we know that 
\begin{equation}\label{conv-Sn}
\frac{S_n}{\sqrt{\mathcal N_{n}/2}}= \( \begin{array}{c}
W_2(n)  \\
W_3(n) \\
W_4(n) \end{array} \) \mathop{\to}^{\rm d} \( \begin{array}{c}
Z_2  \\
Z_3 \\
Z_4 \end{array} \) \sim \mathcal{N}(0,\Sigma)\,, \qquad \Sigma=\Sigma_{\eta}:=\( \begin{array}{ccc}
\frac{3+\eta}{8} & \frac{1-\eta}{8} & 0 \\
\frac{1-\eta}{8} & \frac{3+\eta}{8} & 0 \\
0 & 0 & \frac{1-\eta}{8} \end{array} \).
\end{equation}
Note that $\eta\in[-1,1]$, and $\det (\Sigma)\ne 0$ if and only if $\eta \in (-1,1)$. Note that $\widetilde S_n$ introduced in Section \ref{ontheproof} coincides with 
$S_n/\sqrt{\mathcal N_{n}/2}$. 

Thanks to the G\"artner-Ellis Theorem (see point (c) of \cite[Theorem 2.3.6]{DZ:98} and Theorem \ref{GET}), in order to establish a MDP for $S_n$, it suffices 
to prove that, for suitable $\{\alpha_n\}_n$, there exists
$$
\lim_{n\sto\infty}\frac{1}{\alpha_n} \log\E\[\exp\(\alpha_n\left \langle\theta,\frac{S_n}{\sqrt{\alpha_n\mathcal N_n/2}}\right \rangle\)\]=: \psi(\theta)\,,
$$
where $\theta=(\theta_1,\theta_2,\theta_3)\in\R^3$ is a three-dimensional vector, and  that
\begin{equation}\label{psi}
\psi(\theta)=\frac12\langle\theta,\Sigma\theta\rangle\,.
\end{equation}
This is an easy task and it is immediate to verify that, if 
\begin{equation}\label{eps}
\alpha_{n}\to+\infty \qquad \text{and} \qquad \frac{\alpha_{n}}{\mathcal N_{n}}\to 0\,,
\end{equation}
then the random vector 
$$\left \lbrace \alpha_n^{-1/2}\frac{S_n}{\sqrt{\mathcal N_{n}/2}}\right \rbrace_n$$ 
satisfies a MDP with rate function 
\begin{equation}\label{ratefunction}
\psi^\ast(x)=\sup_{\theta\in\R^3}\{\langle\theta,x\rangle-\psi(\theta)\}, \quad x\in \mathbb R^3.
\end{equation} 
Indeed, setting for $\lambda\in \Lambda_n^+$,
$$
b_n(\lambda)= \( \begin{array}{c}
\lambda_1^2/n  \\
\lambda_2^2/n  \\
\lambda_1\lambda_2/n \end{array} \),
$$
one has that
\begin{flalign*}
&\frac{1}{\alpha_n} \log\E\[\exp\(\alpha_n\left \langle\theta,\frac{S_n}{\sqrt{\alpha_n\mathcal N_n/2}}\right \rangle\)\]\\
&=\frac{1}{\alpha_n}\sum_{\lambda\in\Lambda_{n_j}^+} \{\frac12\frac{\langle\theta,b_{n}(\lambda)\rangle^2}{{\mathcal N_n /(2\alpha_n)}}+\frac13\frac{\langle\theta,b_{n}(\lambda)\rangle^3}{(\mathcal N_n /(2\alpha_n))^{3/2}}+o\(\frac{\langle\theta,b_{n}(\lambda)\rangle^3}{(\mathcal N_n/(2\alpha_n))^{3/2}}\)\},
\end{flalign*}
where we used the fact that $2\cdot |a_\lambda|^2$ is distributed as a $\chi^2(2)$ random variable\footnote{Recall that $a_\lambda = b_\lambda + i c_\lambda$, where $b_\lambda$ and $c_\lambda$ are iid $\sim \mathcal N(0,1/2).$}.
So our goal becomes proving that
\begin{equation}\label{prova1}
\lim_{j\sto+\infty}\frac{1}{\mathcal N_{n_j}/2}\sum_{\lambda\in\Lambda_{n_j}^+}\langle\theta,b_{n_j}(\lambda)\rangle^2=\langle\theta,\Sigma\theta\rangle\,,
\end{equation}
and that
\begin{equation}\label{prova2}
\frac{1}{\alpha_n}\sum_{\lambda\in\Lambda_{n_j}^+} \frac{\langle\theta,b_{n}(\lambda)\rangle^3}{(\mathcal N_n /\alpha_n)^{3/2}} \to  0\,.
\end{equation}
As far as (\ref{prova1}) is concerned, given the fact that for every $n\in S$ (see \cite[Lemma 4.1]{MPRW:16})
$$
\frac{1}{n^2\mathcal N_{n}}\sum_{\lambda\in\Lambda_n}\lambda_1^4=\frac{1}{n^2\mathcal N_{n}}\sum_{\lambda\in\Lambda_n}\lambda_2^4= \frac{3+\widehat{\mu}_n(4)}{8}
\quad \text{ and } \quad
\frac{1}{n^2\mathcal N_{n}}\sum_{\lambda\in\Lambda_n}\lambda_1^2\lambda_2^2= \frac{1-\widehat{\mu}_n(4)}{8}\,,
$$
we simply have that
\begin{flalign*}
&\frac{1}{\mathcal N_{n}/2}\sum_{\lambda\in\Lambda_n^+}\langle\theta,b_n(\lambda)\rangle^2=\frac{1}{\mathcal N_n/2}\sum_{\lambda\in\Lambda_n^+}\(\theta_1\frac{\lambda_1^2}{n}+\theta_2\frac{\lambda_2^2}{n}+\theta_3\frac{\lambda_1\lambda_2}{n}\)^2\\
&=\theta_1^2\frac{3+\widehat{\mu}_n(4)}{8}
+\theta_2^2\frac{3+\widehat{\mu}_n(4)}{8}+\theta_3^2\frac{1-\widehat{\mu}_n(4)}{8}
+2\theta_1\theta_2\frac{1-\widehat{\mu}_n(4)}{8}
\end{flalign*}
since clearly
$\sum_{\lambda\in\Lambda_n}\lambda_1\lambda_2^3=\sum_{\lambda\in\Lambda_n}\lambda_1^3\lambda_2=0$.

As a consequence we just proved that
\begin{flalign*}
&\frac{1}{\mathcal N_{n}/2}\sum_{\lambda\in\Lambda_n^+}\langle\theta,b_n(\lambda)\rangle^2=\langle\theta,\Sigma_n\theta\rangle\,,
\end{flalign*}
where
$$
\Sigma_n=\( \begin{array}{ccc}
\frac{3+\widehat{\mu}_n(4)}{8} & \frac{1-\widehat{\mu}_n(4)}{8} & 0 \\
\frac{1-\widehat{\mu}_n(4)}{8} & \frac{3+\widehat{\mu}_n(4)}{8} & 0 \\
0 & 0 & \frac{1-\widehat{\mu}_n(4)}{8} \end{array} \).
$$
Therefore, since $\lim_{j\sto\infty}\widehat{\mu}_n(4)=\eta$, we obtain (\ref{prova1}).

As for (\ref{prova2}), we have 
\begin{flalign*}
\frac{1}{\alpha_{n_j}} \sum_{\lambda\in\Lambda_{n_j}^+} \frac{\langle\theta,b_{n_j}(\lambda)\rangle^3}{\({\mathcal N_{n_j}}/{\alpha_{n_j}}\)^{3/2}} &= \frac{1}{\,\mathcal N_{n_j}^{3/2}\alpha_{n_j}^{-1/2}}\sum_{\lambda\in\Lambda_{n_j}^+} {\langle\theta,b_{n_j}(\lambda)\rangle^3}
\le c \frac{\| \theta\|^3}{\,\mathcal N_{n_j}^{1/2}\alpha_{n_j}^{-1/2}}
\end{flalign*}
(for some absolute constant $c>0$),
which tends to $0$ as $j\sto\infty$, since $\alpha_{n_j}$ is chosen as in \eqref{eps}.


\subsubsection{Exponential equivalence}\label{step2}

In light of (\ref{4chaosred}) and (\ref{varesatta4chaos}), let us define, slightly abusing notation,
\begin{equation}\label{Mn}
\mathcal M_n= \mathcal M_{\widehat \mu_n(4)} :=\(1+\wh \mu_n (4)^2+\frac{34}{\mathcal N_n}\)^{-1/2}\((W_2(n)+W_3(n))^2-2W_2(n)^2-2W_3(n)^2-4W_4(n)^2\).
\end{equation}
\begin{lemma}\label{lemmaEE}
Consider as usual $\{n_j\}_j\subseteq S$ such that $\mathcal N_{n_j}\sto+\infty$ as $j\sto+\infty$ and 
assume that
$$
\alpha_{n_j} \to +\infty \quad \text{and} \quad \frac{\alpha_{n_j}}{\log \mathcal N_{n_j}}\to 0\,,
$$
then, for every $\delta>0$,
$$
\limsup_{j\sto+\infty}\frac1{\alpha_{n_j}} \log\P\(\alpha_{n_j}^{-1}\abs{\widetilde\Lc_{n_j}-\mathcal M_{n_j}}>\delta\)=-\infty\,,
$$
$\text{i.e.}$ $\alpha_{n_j}^{-1}\widetilde\Lc_{n_j}$ and $\alpha_{n_j}^{-1}\mathcal M_{n_j}$ are exponentially equivalent.
\end{lemma}

\proof 
From Lemma \ref{p:formula4} and (\ref{varesatta4chaos}) we can write 
\begin{flalign*}
\widetilde \Lc_n[4] 
=\mathcal M_n +  \(1+\wh \mu_n (4)^2+\frac{34}{\mathcal N_n}\)^{-1/2}R(n)
\end{flalign*}
where we recall that
$$
R(n)=\frac12 \frac{1}{\mathcal N_n/2}\sum_{\lambda \in \Lambda_n^+} \abs{a_\lambda}^4
$$ 
is a sequence of (independent) random variables converging in probability to $1$ (note that $\mathbb E[R(n)]=1$). We have, for $n$ large enough,
\begin{flalign*}
&\P\(\alpha_{n}^{-1}\abs{\widetilde\Lc_{n}-\mathcal M_{n}}>\delta\) =\P\(\alpha_{n}^{-1} \abs{\widetilde\Lc_{n}-\widetilde\Lc_{n}[4] +  \(1+\wh \mu_n (4)^2+\frac{34}{\mathcal N_n}\)^{-1/2}R(n)} > \delta\)\\
&\le \P\(\alpha_{n}^{-1} \abs{\widetilde\Lc_{n}-\widetilde\Lc_{n}[4]} > \frac{\delta}{2}\) + \P\(\alpha_{n}^{-1}  \(1+\wh \mu_n (4)^2+\frac{34}{\mathcal N_n}\)^{-1/2}\left | R(n) - 1\right | > \frac{\delta}{4}\)
\end{flalign*}
since $\alpha_{n}^{-1}  \(1+\wh \mu_n (4)^2+\frac{34}{\mathcal N_n}\)^{-1/2}\to 0$ as $\mathcal N_n\to +\infty$. 
As for the first summand, thanks to Markov inequality, we can write
\begin{flalign*}
\P\(\alpha_{n}^{-1}\abs{\widetilde\Lc_{n}-\mathcal M_{n}}>\delta/2\)&\le \frac{\alpha_{n}^{-1}}{\delta/2}\E\[ \abs{\widetilde\Lc_{n}-{\mathcal M}_{n}}	\]\\
&\le \frac{\alpha_{n}^{-1}}{\delta/2}\E\[ \abs{\widetilde\Lc_{n}-\widetilde\Lc_{n}[4]}^2 \]^{\frac12}
\end{flalign*}
and from \eqref{PR-bound} we have 
$$
\E\[ \abs{\widetilde\Lc_{n}-\widetilde\Lc_{n}[4]}^2\]=O\(\mathcal  N_n^{-1/2}\).
$$ Since $\alpha_n/\log \mathcal N_n \to 0$ we deduce 
\begin{flalign*}
&\limsup_{n\sto+\infty}\frac1{\alpha_{n}} \log\P\(\alpha_{n}^{-1}\abs{\widetilde\Lc_{n}-\wh{\mathcal M}_{n}}>\delta/2\) \le \limsup_{n\sto+\infty}\frac1{\alpha_{n}} \log\frac{\alpha_{n}^{-1}}{\delta/2}\E\[ \abs{\widetilde\Lc_{n}-\Lc_{n}[4]}^2 \]^\frac12 = -\infty. 
\end{flalign*}
Regarding the second summand, from \cite[(85)]{PR:18},
\begin{flalign*}
\P\( \left | R(n) - 1\right | > \frac{\delta}{4}\alpha_n \(1+\wh \mu_n (4)^2+\frac{34}{\mathcal N_n}\)^{1/2}\) &\le \frac{\Var(R(n))}{\alpha_n^2\delta^2/(16) \(1+\wh \mu_n (4)^2+\frac{34}{\mathcal N_n}\)}\\
&\ll \frac{1}{\alpha_n^2\mathcal N_n}
\end{flalign*}
so that, since $\alpha_n/\log \mathcal N_n \to 0$ 
\begin{flalign*}
\limsup_{n\to +\infty} \frac{1}{n} \log \P\( \left | R(n) - 1\right | > \frac{\delta}{4}\alpha_n \(1+\wh \mu_n (4)^2+\frac{34}{\mathcal N_n}\)^{1/2}\) = -\infty\end{flalign*}
thus concluding the proof of the exponential equivalence.
\endproof


\subsubsection{Contraction principle}\label{step3}

In light of the exponential equivalence obtained in the previous section, in order to obtain a MDP for the nodal length of arithmetic random waves, it suffices to prove the following. 
\begin{lemma}\label{lemmaM}
The sequence of random variables $\lbrace \alpha_n^{-1}\mathcal M_n\rbrace_n$, where $\mathcal N_n\to +\infty$ such that $\widehat \mu_n(4)\to \eta\in [-1,1]$, $\mathcal M_n$ is defined as in (\ref{Mn}), and $\lbrace \alpha_n\rbrace_n$ is a sequence of positive numbers satisfying (\ref{eps}),
enjoys a MDP with speed $\alpha_n$ and rate function
\begin{equation}\label{rfM}
I_\eta(y)=\begin{cases}
-y\frac{\sqrt{1 + \eta^2}}{1+|\eta|},\quad &y\le 0\\
+\infty,\quad &y>0.
\end{cases} 
\end{equation}
\end{lemma}
\begin{proof}
First of all we note that 
\begin{eqnarray}\label{alphaM}
&\alpha_n^{-1}\mathcal M_n = \(1+\wh \mu_n (4)^2+\frac{34}{\mathcal N_n}\)^{-1/2}f\(\frac{S_n}{\sqrt{\alpha_n\mathcal N_{n}/2}}\)\nonumber\\
&=\(1+\wh \mu_n (4)^2+\frac{34}{\mathcal N_n}\)^{-1/2} f\(\alpha_n^{-1/2}\(W_2(n),W_3(n),W_4(n)\)^\prime\),
\end{eqnarray} 
where $S_n$ is defined as in (\ref{defSn}) and $f:\mathbb R^3\to \mathbb R$ is given by 
\begin{flalign*}
f(x_1,x_2,x_3)&=(x_1+x_2)^2-2x_1^2-2x_2^2-4x_3^2\\
&=-(x_1-x_2)^2-4x_3^2 \,.
\end{flalign*}
Thanks to the contraction principle (see \cite[Theorem 4.2.1]{DZ:98} and Theorem \ref{contractionprinciple}), the sequence of random variables 
$$\left \lbrace f\(\frac{S_n}{\sqrt{\alpha_n\mathcal N_{n}/2}}\)\right \rbrace_n$$ 
satisfies a LDP with speed  $\alpha_n$ and rate function 
$$
I_f(y)=\inf\{\psi^\ast(x):\,\,f(x)=y\}, \quad y\in \mathbb R\,,
$$
where $\psi^\ast$ is defined as in (\ref{ratefunction}).  Since $f$ takes non-positive values,
$$
I_f(y)=+\infty \qquad \text{whenever}  \quad y>0
$$
and 
$$
I_f(0)=0 \qquad \text{since} \quad \psi^\ast(x)=0 \iff x=0 \quad \text{and} \quad f(0)=0.
$$
We are going to explicitly compute $I_f(y)$ for $y < 0$. Assume that $\eta\in (-1,1)$, then the matrix $\Sigma$ in \eqref{conv-Sn}  is non-singular and 
$$
I_f(y)=\inf\{\psi^\ast(x):\,\,f(x)=y\}=\inf\{\frac12\langle x,\Sigma^{-1}  x\rangle:\,\,f(x)=y\}.
$$
Now 
$$
\Sigma^{-1}=\frac{8}{1+\eta}\( \begin{array}{ccc}
\frac{3+\eta}{8} & \frac{\eta-1}{8} & 0 \\
\frac{\eta-1}{8} & \frac{3+\eta}{8} & 0 \\
0 & 0 & \frac{1+\eta}{1-\eta} \end{array} \),
$$
and therefore
$$
\langle x,\Sigma^{-1}x\rangle=\frac{8}{1+\eta}\(\frac{3+\eta}{8}(x_1^2+x_2^2)+\frac{\eta-1}{4}x_1x_2+\frac{1+\eta}{1-\eta}x_3^2\)\,.
$$
As a consequence,
$$
\psi^\ast(x)=\frac12\langle x,\Sigma^{-1}  x\rangle=\frac{4}{1+\eta}\(\frac{3+\eta}{8}(x_1^2+x_2^2)+\frac{\eta-1}{4}x_1x_2+\frac{1+\eta}{1-\eta}x_3^2\)
$$
and
\begin{flalign}
I_f(y)&=\inf\{\frac{4}{1+\eta}\(\frac{3+\eta}{8}(x_1^2+x_2^2)+\frac{\eta-1}{4}x_1x_2+\frac{1+\eta}{1-\eta}x_3^2\):\,-x_1^2-x_2^2+2x_1x_2-4x_3^2=y\}\,. \label{inf-eta}
\end{flalign}
Let us now compute $I_f(y)$ for any $\eta\in(-1,1)$. We will use the Lagrange multipliers method with
$$
L(x_1,x_2,x_3,\lambda)=h(x_1,x_2,x_3)+\lambda g(x_1,x_2,x_3)
$$
where
\begin{flalign*}
h(x_1,x_2,x_3)&=\frac{3+\eta}{2(1+\eta)}(x_1^2+x_2^2)+\frac{(\eta-1)}{1+\eta}x_1x_2+\frac{4}{1-\eta}x_3^2\\
g(x_1,x_2,x_3)&=-x_1^2-x_2^2+2x_1x_2-4x_3^2-y\,.
\end{flalign*}
Then, recalling that we take $y<0$, we have the system
\begin{flalign*}
\begin{cases}
h_{x_1}+\lambda g_{x_1}=0\\
h_{x_2}+\lambda g_{x_2}=0\\
h_{x_3}+\lambda g_{x_3}=0\\
g(x_1,x_2,x_3)=0\textcolor{red}{,}
\end{cases}
\text{ which yields } \,\,
\begin{cases}
\ds\lambda=\frac{1}{1+\eta}\\[10pt]
\ds  x_1=\mp\sqrt{-\frac{y}{4}}\\[10pt]
\ds  x_2=\pm\sqrt{-\frac{y}{4}}\\[10pt]
x_3=0
\end{cases}
\text{ or }
\qquad
\begin{cases}
\ds\lambda=\frac{1}{1-\eta}\\[10pt]
\ds x_1=0\\[10pt]
\ds x_2=0\\[10pt]
\ds x_3=\pm\sqrt{-\frac{y}{4}}\textcolor{red}{.}
\end{cases}
\end{flalign*}
For the first two solutions we have that
$$
\psi^\ast\(\sqrt{-\frac{y}{4}}, -\sqrt{-\frac{y}{4}},0\)=\psi^\ast\(-\sqrt{-\frac{y}{4}},\sqrt{-\frac{y}{4}},0\)=\frac{-y}{1+\eta}\,,
$$
while for the other two solutions we have that
$$
\psi^\ast\(0,0,\pm\sqrt{-\frac{y}{4}}\)=\frac{-y}{1-\eta}\,.
$$
As a consequence, the rate function is given by 
$$
I_f(y)=\frac{-y}{1+\abs{\eta}}\,,  \quad \text{ for } \quad y\le0 \quad\text{ and } \quad \eta\in(-1,1)\,.
$$ 

If $\eta\in \{-1,1\}$, then $\Sigma$ is singular. In this case, in order to  compute  $\psi^\ast(x)$ one has to consider the gradient with respect to $\theta$ to obtain the supremum in \eqref{ratefunction} and it is a known fact that one obtains $\Sigma\theta=x$. Then there are  two possibilities:
\begin{itemize}
\item if $x\notin \Ima(\Sigma)$, then $\Sigma\theta=x$ has no solution, and therefore $\psi^\ast(x)=+\infty$;
\item if $x\in \Ima(\Sigma)$, then $\Sigma\theta=x$ has solution $\theta=\wt\Sigma^{-1}x$ where $\wt\Sigma^{-1}$ is the inverse of $\Sigma$ restricted to $\Ima(\Sigma)$, and therefore $\psi^\ast(x)=\[\langle\theta,x\rangle-\frac12\langle\theta,\Sigma\theta\rangle\]_{\theta=\wt\Sigma^{-1}x}=\frac12\langle x,\wt\Sigma^{-1}x\rangle$.
\end{itemize}
Then, considering the case when $\eta=1$, we have:
$$
\Sigma=\( \begin{array}{ccc} 
\frac12 & 0 & 0 \\
0 & \frac12 & 0 \\
0 & 0 & 0 \end{array} \)
$$
and hence $\Ima(\Sigma)=\{(x_1,x_2,x_3):\,x_3=0\}$. As a consequence, for $x\in\Ima(\Sigma)$, we have that
$$
\( \begin{array}{ccc} 
\frac12 & 0 & 0 \\
0 & \frac12 & 0 \\
0 & 0 & 0 \end{array} \)\( \begin{array}{c} 
\theta_1 \\
\theta_2 \\
\theta_3 \end{array} \)=\( \begin{array}{c} 
x_1 \\
x_2 \\
0 \end{array} \),\text{ which yields   }\,\,
\theta=\( \begin{array}{ccc} 
2 & 0 & 0 \\
0 & 2 & 0 \\
0 & 0 & 0 \end{array} \)x,
$$
and hence
$$
\psi^\ast(x)=\frac12\langle x,\wt\Sigma^{-1}x\rangle=x_1^2+x_2^2.
$$
So what we have to compute now is the following rate function
\begin{flalign*}
I_f(y)&=\inf_{x\in\Ima(\Sigma)}\{\frac12\langle x,\wt\Sigma^{-1}  x\rangle:\,\,f(x)=y\}=\inf_{x\in\Ima(\Sigma)}\{x_1^2+x_2^2:\,-x_1^2-x_2^2+2x_1x_2=y\}.
\end{flalign*}
We use again the Lagrange method. We have the system
$$
\begin{cases}
2x_1+\lambda(-2x_1+2x_2)=0\\
2x_2+\lambda(-2x_2+2x_1)=0\\
-x_1^2-x_2^2+2x_1x_2-y=0\textcolor{red}{,}
\end{cases}
\text{which yields } \begin{cases}
x_1=\pm\sqrt{-\frac y4}\\
x_2=\mp\sqrt{-\frac y4}\\
\lambda=\frac12
\end{cases}
\text{  and } \psi^\ast\(\pm \sqrt{-\frac y4}, \mp \sqrt{-\frac y4}, 0\)=-\frac y 2\,.
$$
Thus, for $\eta=1$, the rate function is $I_f(y)=-y/2$. 


Let us now consider the case $\eta=-1$. We have
$$
\Sigma=\( \begin{array}{ccc} 
\frac14 & \frac14 & 0 \\
\frac14 & \frac14 & 0 \\
0 & 0 & \frac14 \end{array} \)
$$
and hence $\Ima(\Sigma)=\{(x_1,x_2,x_3):\,x_1=x_2\}$. As a consequence, for $x\in\Ima(\Sigma)$,
$$
\( \begin{array}{ccc} 
\frac14 & \frac14 & 0 \\
\frac14 & \frac14 & 0 \\
0 & 0 & \frac14 \end{array} \)\( \begin{array}{c} 
\theta_1 \\
\theta_2 \\
\theta_3 \end{array} \)=\( \begin{array}{c} 
x_1 \\
x_1 \\
x_3 \end{array} \)\text{which yields } \begin{cases}\frac14\theta_1+\frac14\theta_2=x_1\\  \frac14\theta_1+\frac14\theta_2=x_1\\\frac14\theta_3=x_3 \end{cases}\text{which yields } \begin{cases} \theta_1+\theta_2=4x_1\\  \theta_1+\theta_2=4x_1\\\theta_3=4x_3 \end{cases}
$$
and therefore
\begin{flalign*}
\psi^\ast(x_1,x_1,x_3)&=\[\langle\theta,x\rangle-\frac12\langle\theta,\Sigma\theta\rangle\]_{\Sigma\theta=x,x_1=x_2}
=2x_1^2+2x_3^2.
\end{flalign*}
So what we have to compute now is the following rate function
\begin{flalign*}
I_f(y)&=\inf_{x\in\Ima(\Sigma)}\{2x_1^2+2x_3^2:\,-4x_3^2=y\}.
\end{flalign*}
We use again the Lagrange method. We have the system
$$
\begin{cases}
4x_1+\lambda\cdot0=0\\
4x_3+\lambda(-8x_3)=0\\
-4x_3^2=y\textcolor{red}{,}
\end{cases}
\text{which yields }
\begin{cases}
x_1=0\\
x_3=\pm\sqrt{-\frac y 4}\\
\lambda=
\end{cases}
\text{  and } \quad \psi^\ast\(0,0,\pm \sqrt{-\frac y4}\)=-\frac y 2\,.
$$
Thus, for $\eta=-1$, the rate function is $I_f(y)=-y/2$. \\

Finally, recalling (\ref{alphaM}), since as $n\sto +\infty$
$$
\left ( 1 + \widehat \mu_n(4)^2 + \frac{34}{\mathcal N_n}\right ) \to 1 + \eta^2,
$$
one can prove the desired LDP of $\lbrace \alpha_n^{-1} \mathcal
M_n\rbrace_n$, with the rate function in (\ref{rfM}), with some standard
computations (in particular, for the proof of the lower bound for open sets,
one can use Lemma 19 in \cite{GT:08}).
\end{proof}\\

\begin{proof}[Proof of Theorem   \ref{MR}]
Bearing in mind the result of Sections \ref{step1}--\ref{step3}, the proof of Theorem \ref{MR} immediately follows. 
Indeed, from Lemma \ref{lemmaEE} the two sequences of random variables 
$
\{\widetilde\Lc_{n}/\alpha_{n}\}_n$ and $\{\mathcal M_{n}/\alpha_{n}\}_n$  are exponential equivalent at speed $\alpha_n$. From Theorem \ref{thDZ} and Lemma \ref{lemmaM}, $
\{\widetilde\Lc_{n}/\alpha_{n}\}_n$ enjoys a MDP with speed $\alpha_n$ and rate function (\ref{rfM}). 
\end{proof}

\subsection{Proof of Theorem \ref{MR-shrinking}}

\begin{proof}
The proof consists in showing that $\alpha_n^{-1}\widetilde\Lc_{n;s}$ and $\alpha_n^{-1} \widetilde\Lc_{n}$ are exponentially equivalent at speed $\alpha_n$: this guarantees that they enjoy the same MDP as in Theorem \ref{MR}. In order to do that, it suffices to find an upper bound for
\begin{flalign*}
\E\[\abs{\widetilde\Lc_{n;s}-\widetilde\Lc_{n}}^2\]=1+1-2 \, \Corr(\Lc_{n;s},\Lc_{n}).
\end{flalign*}
Recalling (\ref{eq:var leading KKW}), (\ref{eq:Var asympt gen}) and the precise identity \cite[(3.36)]{BMW:20}, we can write, uniformly for $s> n^{-\frac12 +\eps}$,
\begin{flalign*}
\Corr(\Lc_{n;s},\Lc_{n})&=\frac{\Cov(\Lc_{n;s},\Lc_{n})}{\sqrt{\Var\(\Lc_{n;s}\)\Var\(\Lc_{n}\)}}=	\frac{\pi s^2 \Var\(\Lc_{n}\)}{\sqrt{\Var\(\Lc_{n;s}\)\Var\(\Lc_{n}\)}}\\
&=	\frac{\pi s^2 \sqrt{\Var\(\Lc_{n}\)}}{\sqrt{\Var\(\Lc_{n;s}\)}}=\frac{\pi s^2 \sqrt{\frac{E_n}{\mathcal N_n^2} c_n+O\(\frac{E_n}{\mathcal N_n^{5/2}}\)}}{\sqrt{(\pi s^2)^2 \(\frac{E_n}{\mathcal N_n^2} c_n+O_{\eps}\(\frac{E_n}{\mathcal N_n^{5/2}}\)\)}}\\
&=1+O_\eps\(\mathcal N_n^{-1/2}\)\,,
\end{flalign*}
where the constant involved in the $O$-notation only depends on $\eps$. 
As a  consequence, we have that
\begin{equation}\label{boundAnna}
\E\[\abs{\widetilde\Lc_{n;s}-\widetilde\Lc_{n}}^2\]=O_\eps\(\mathcal N_n^{-1/4}\)\,.
\end{equation}
Now, for every $\delta>0$,  thanks to (\ref{boundAnna}) we have 
\begin{flalign*}
&\limsup_{n\sto+\infty}\frac1{\alpha_n} \log\P\(\alpha_n^{-1}\abs{\widetilde\Lc_{n;s}-\widetilde\Lc_{n}}>\delta\) \le \limsup_{n\sto+\infty}\frac1{\alpha_n} \log \frac{\alpha_n^{-2}}{\delta^2}\E\[\abs{\widetilde\Lc_{n;s}-\widetilde\Lc_{n}}^2\] = -\infty
\end{flalign*}
and the proof is hence concluded.
\end{proof}

\subsection*{Acknowledgements} 
The authors wish to thank an anonymous referee for his/her suggestions that greatly improved the presentation of this paper. 

C.~M. has been supported by the MIUR Excellence Department Project awarded to the Department of Mathematics, 
University of Rome ``Tor Vergata'' (CUP E83C18000100006 and CUP E83C23000330006), by
University of Rome ``Tor Vergata'' (``Asymptotic Properties in Probability'' (CUP E83C22001780005)) and by the GNAMPA-INdAM 
Project \emph{Stime asintotiche: principi di invarianza e grandi deviazioni}. M.~R. has been supported by the 
ANR-17-CE40-0008 project \emph{Unirandom}. A.~V. has been supported by the co-financing of the European Union - FSE-REACT-EU, 
PON Research and Innovation 2014--2020, DM 1062/2021.

 \bibliographystyle{alpha}

\end{document}